\documentclass[12pt, a4paper, reqno]{amsart}
\usepackage{a4wide}
\usepackage{enumerate}

\newcommand{\shin}[1]{}
\newcommand{\seo}[1]{}
\usepackage{color}
\usepackage{ifpdf}
\ifpdf %if using pdfLaTeX in PDF mode
  \usepackage[pdftex]{graphicx}
  \DeclareGraphicsExtensions{.pdf,.png,.mps}
  \usepackage{pgf}
  \usepackage{tikz}
\else %if using LaTeX or pdfLaTeX in DVI mode
  \usepackage{graphicx}
  \DeclareGraphicsExtensions{.eps,.bmp}
  \usepackage{pgf}
  \usepackage{tikz}
  \usepackage{pstricks}%variant: \usepackage{pst-all}
\fi
\usepackage{epic,bez123}
\usepackage{floatflt}% package for floatingfigure environment
\usepackage{wrapfig}% package for wrapfigure environment

\usepackage{amsthm}

\theoremstyle{plain}
\newtheorem{thm}{Theorem}%[section]

\newtheorem{lem}[thm]{Lemma}

\theoremstyle{definition}

\theoremstyle{remark}

\usepackage{amsmath}
\DeclareMathOperator{\MD}{MD}

\newcommand{\abs}[1]{\left|#1\right|}
\newcommand{\set}[1]{\left\{#1\right\}}

\newcommand{\Ord}{\mathcal{O}}
\newcommand{\Z}{\mathcal{Z}}
\newcommand{\F}{\mathcal{F}}

\title{A refinement for ordered labeled trees}
\author{Seunghyun Seo}
\address[Seunghyun Seo]{Department of Mathematics Education, Kangwon National University, Chuncheon 200-701, Korea}
\email{shyunseo@kangwon.ac.kr}

\author{Heesung Shin$^\dag$}
\address[Heesung Shin]{Department of Mathematics, Inha University, Incheon 402-751, Korea}
\email{shin@inha.ac.kr}
\thanks{\dag Corresponding author}


\date{\today}

\begin{document}
\maketitle
\begin{abstract}
Let $\Ord_n$ be the set of ordered labeled trees on $\{0,\dots,n\}$. A maximal decreasing subtree of an ordered labeled tree
is defined by the maximal ordered subtree from the root with all edges being decreasing. In this paper, we study a new refinement $\Ord_{n,k}$ of $\Ord_n$, which is the set of ordered labeled trees whose maximal decreasing subtree has $k+1$ vertices.
\end{abstract}
%\tableofcontents

\section{Introduction}
An \emph{ordered tree} is a rooted tree in which children of each vertex are ordered. Figure~\ref{fig:ord} shows all
the ordered tree with $4$ vertices. It is well known (see \cite[Exercise 6.19]{Sta99}) that the number of ordered trees with $n+1$ vertices is given by the $n$th Catalan number $C_n = \frac{1}{n+1} \binom{2n}{n}$.

\begin{figure}[h]
%<TpX v="5" PdfTeXFormat="pgf" ArrowsSize="0.7" StarsSize="1" DefaultFontHeight="4" DefaultSymbolSize="30" ApproximationPrecision="0.01" PicScale="1" Border="2" BitmapRes="20000" HatchingStep="2" DottedSize="0.5" DashSize="1" LineWidth="0.3" TeXCenterFigure="1" TeXFigure="none">
%  <star x="10" y="10"/>
%  <star x="10" y="20"/>
%  <star x="10" y="30"/>
%  <star x="10" y="40"/>
%  <star x="40" y="30"/>
%  <star x="30" y="30"/>
%  <star x="30" y="20"/>
%  <star x="50" y="30"/>
%  <star x="35" y="40"/>
%  <star x="55" y="40"/>
%  <star x="60" y="30"/>
%  <star x="60" y="20"/>
%  <star x="80" y="40"/>
%  <star x="80" y="30"/>
%  <star x="75" y="20"/>
%  <star x="85" y="20"/>
%  <star x="110" y="40"/>
%  <star x="110" y="30"/>
%  <star x="100" y="30"/>
%  <star x="120" y="30"/>
%  <line x1="10" y1="40" x2="10" y2="30"/>
%  <line x1="10" y1="30" x2="10" y2="20"/>
%  <line x1="10" y1="20" x2="10" y2="10"/>
%  <line x1="35" y1="40" x2="30" y2="30"/>
%  <line x1="30" y1="30" x2="30" y2="20"/>
%  <line x1="35" y1="40" x2="40" y2="30"/>
%  <line x1="55" y1="40" x2="50" y2="30"/>
%  <line x1="55" y1="40" x2="60" y2="30"/>
%  <line x1="60" y1="30" x2="60" y2="20"/>
%  <line x1="80" y1="40" x2="80" y2="30"/>
%  <line x1="80" y1="30" x2="75" y2="20"/>
%  <line x1="80" y1="30" x2="85" y2="20"/>
%  <line x1="110" y1="40" x2="100" y2="30"/>
%  <line x1="110" y1="40" x2="110" y2="30"/>
%  <line x1="110" y1="40" x2="120" y2="30"/>
%</TpX>
\centering
\ifpdf
\begin{pgfpicture}{7.00mm}{7.00mm}{123.00mm}{43.00mm}
\pgfsetxvec{\pgfpoint{1.00mm}{0mm}}
\pgfsetyvec{\pgfpoint{0mm}{1.00mm}}
\color[rgb]{0,0,0}\pgfsetlinewidth{0.30mm}\pgfsetdash{}{0mm}
\pgfcircle[fill]{\pgfxy(10.00,10.00)}{1.00mm}
\pgfcircle[stroke]{\pgfxy(10.00,10.00)}{1.00mm}
\pgfcircle[fill]{\pgfxy(10.00,20.00)}{1.00mm}
\pgfcircle[stroke]{\pgfxy(10.00,20.00)}{1.00mm}
\pgfcircle[fill]{\pgfxy(10.00,30.00)}{1.00mm}
\pgfcircle[stroke]{\pgfxy(10.00,30.00)}{1.00mm}
\pgfcircle[fill]{\pgfxy(10.00,40.00)}{1.00mm}
\pgfcircle[stroke]{\pgfxy(10.00,40.00)}{1.00mm}
\pgfcircle[fill]{\pgfxy(40.00,30.00)}{1.00mm}
\pgfcircle[stroke]{\pgfxy(40.00,30.00)}{1.00mm}
\pgfcircle[fill]{\pgfxy(30.00,30.00)}{1.00mm}
\pgfcircle[stroke]{\pgfxy(30.00,30.00)}{1.00mm}
\pgfcircle[fill]{\pgfxy(30.00,20.00)}{1.00mm}
\pgfcircle[stroke]{\pgfxy(30.00,20.00)}{1.00mm}
\pgfcircle[fill]{\pgfxy(50.00,30.00)}{1.00mm}
\pgfcircle[stroke]{\pgfxy(50.00,30.00)}{1.00mm}
\pgfcircle[fill]{\pgfxy(35.00,40.00)}{1.00mm}
\pgfcircle[stroke]{\pgfxy(35.00,40.00)}{1.00mm}
\pgfcircle[fill]{\pgfxy(55.00,40.00)}{1.00mm}
\pgfcircle[stroke]{\pgfxy(55.00,40.00)}{1.00mm}
\pgfcircle[fill]{\pgfxy(60.00,30.00)}{1.00mm}
\pgfcircle[stroke]{\pgfxy(60.00,30.00)}{1.00mm}
\pgfcircle[fill]{\pgfxy(60.00,20.00)}{1.00mm}
\pgfcircle[stroke]{\pgfxy(60.00,20.00)}{1.00mm}
\pgfcircle[fill]{\pgfxy(80.00,40.00)}{1.00mm}
\pgfcircle[stroke]{\pgfxy(80.00,40.00)}{1.00mm}
\pgfcircle[fill]{\pgfxy(80.00,30.00)}{1.00mm}
\pgfcircle[stroke]{\pgfxy(80.00,30.00)}{1.00mm}
\pgfcircle[fill]{\pgfxy(75.00,20.00)}{1.00mm}
\pgfcircle[stroke]{\pgfxy(75.00,20.00)}{1.00mm}
\pgfcircle[fill]{\pgfxy(85.00,20.00)}{1.00mm}
\pgfcircle[stroke]{\pgfxy(85.00,20.00)}{1.00mm}
\pgfcircle[fill]{\pgfxy(110.00,40.00)}{1.00mm}
\pgfcircle[stroke]{\pgfxy(110.00,40.00)}{1.00mm}
\pgfcircle[fill]{\pgfxy(110.00,30.00)}{1.00mm}
\pgfcircle[stroke]{\pgfxy(110.00,30.00)}{1.00mm}
\pgfcircle[fill]{\pgfxy(100.00,30.00)}{1.00mm}
\pgfcircle[stroke]{\pgfxy(100.00,30.00)}{1.00mm}
\pgfcircle[fill]{\pgfxy(120.00,30.00)}{1.00mm}
\pgfcircle[stroke]{\pgfxy(120.00,30.00)}{1.00mm}
\pgfmoveto{\pgfxy(10.00,40.00)}\pgflineto{\pgfxy(10.00,30.00)}\pgfstroke
\pgfmoveto{\pgfxy(10.00,30.00)}\pgflineto{\pgfxy(10.00,20.00)}\pgfstroke
\pgfmoveto{\pgfxy(10.00,20.00)}\pgflineto{\pgfxy(10.00,10.00)}\pgfstroke
\pgfmoveto{\pgfxy(35.00,40.00)}\pgflineto{\pgfxy(30.00,30.00)}\pgfstroke
\pgfmoveto{\pgfxy(30.00,30.00)}\pgflineto{\pgfxy(30.00,20.00)}\pgfstroke
\pgfmoveto{\pgfxy(35.00,40.00)}\pgflineto{\pgfxy(40.00,30.00)}\pgfstroke
\pgfmoveto{\pgfxy(55.00,40.00)}\pgflineto{\pgfxy(50.00,30.00)}\pgfstroke
\pgfmoveto{\pgfxy(55.00,40.00)}\pgflineto{\pgfxy(60.00,30.00)}\pgfstroke
\pgfmoveto{\pgfxy(60.00,30.00)}\pgflineto{\pgfxy(60.00,20.00)}\pgfstroke
\pgfmoveto{\pgfxy(80.00,40.00)}\pgflineto{\pgfxy(80.00,30.00)}\pgfstroke
\pgfmoveto{\pgfxy(80.00,30.00)}\pgflineto{\pgfxy(75.00,20.00)}\pgfstroke
\pgfmoveto{\pgfxy(80.00,30.00)}\pgflineto{\pgfxy(85.00,20.00)}\pgfstroke
\pgfmoveto{\pgfxy(110.00,40.00)}\pgflineto{\pgfxy(100.00,30.00)}\pgfstroke
\pgfmoveto{\pgfxy(110.00,40.00)}\pgflineto{\pgfxy(110.00,30.00)}\pgfstroke
\pgfmoveto{\pgfxy(110.00,40.00)}\pgflineto{\pgfxy(120.00,30.00)}\pgfstroke
\end{pgfpicture}%
\else
  \setlength{\unitlength}{1bp}%
  \begin{picture}(328.82, 102.05)(0,0)
  \put(0,0){\includegraphics{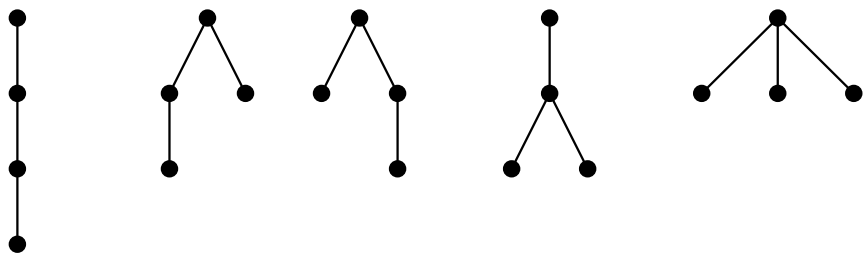}}
  \end{picture}%
\fi
\caption{All ordered trees with $4$ vertices}
\label{fig:ord}
\end{figure}

An \emph{ordered labeled tree} is an ordered tree whose vertices are labeled by distinct nonnegative integers. In most cases, an ordered labeled tree with $n+1$ vertices is identified with an ordered tree on the vertex set $[0,n]:=\{0,\dots,n\}$.
Let $\Ord_n$ be the set of ordered labeled trees on $[0,n]$. Clearly the cardinality of $\Ord_n$ is given by
\begin{equation} \label{eq:cat}
|\Ord_n| = (n+1)!\, C_n = (n+1)^{(n)},
\end{equation}
where $m^{(k)}:=m(m+1)\cdots(m+k-1)$ is a rising factorial.

For a given ordered labeled tree $T$, a \emph{maximal decreasing subtree} of $T$ is defined by the maximal ordered subtree from the root with all edges being decreasing, denoted by $\MD(T)$. Figure~\ref{fig:tree} illustrates the maximal decreasing subtree of a given tree $T$. Let $\Ord_{n,k}$ be the set of ordered labeled trees on $[0,n]$ with its maximal decreasing subtree having $k$ edges.
\begin{figure}[t]
\centering
\ifpdf
\begin{pgfpicture}{17.00mm}{16.14mm}{126.00mm}{54.14mm}
\pgfsetxvec{\pgfpoint{1.00mm}{0mm}}
\pgfsetyvec{\pgfpoint{0mm}{1.00mm}}
\color[rgb]{0,0,0}\pgfsetlinewidth{0.30mm}\pgfsetdash{}{0mm}
\pgfcircle[fill]{\pgfxy(40.00,50.00)}{1.00mm}
\pgfcircle[stroke]{\pgfxy(40.00,50.00)}{1.00mm}
\pgfcircle[fill]{\pgfxy(30.00,40.00)}{1.00mm}
\pgfcircle[stroke]{\pgfxy(30.00,40.00)}{1.00mm}
\pgfcircle[fill]{\pgfxy(40.00,40.00)}{1.00mm}
\pgfcircle[stroke]{\pgfxy(40.00,40.00)}{1.00mm}
\pgfcircle[fill]{\pgfxy(50.00,40.00)}{1.00mm}
\pgfcircle[stroke]{\pgfxy(50.00,40.00)}{1.00mm}
\pgfcircle[fill]{\pgfxy(40.00,30.00)}{1.00mm}
\pgfcircle[stroke]{\pgfxy(40.00,30.00)}{1.00mm}
\pgfcircle[fill]{\pgfxy(60.00,30.00)}{1.00mm}
\pgfcircle[stroke]{\pgfxy(60.00,30.00)}{1.00mm}
\pgfcircle[fill]{\pgfxy(30.00,30.00)}{1.00mm}
\pgfcircle[stroke]{\pgfxy(30.00,30.00)}{1.00mm}
\pgfcircle[fill]{\pgfxy(50.00,20.00)}{1.00mm}
\pgfcircle[stroke]{\pgfxy(50.00,20.00)}{1.00mm}
\pgfcircle[fill]{\pgfxy(70.00,20.00)}{1.00mm}
\pgfcircle[stroke]{\pgfxy(70.00,20.00)}{1.00mm}
\pgfcircle[fill]{\pgfxy(40.00,20.00)}{1.00mm}
\pgfcircle[stroke]{\pgfxy(40.00,20.00)}{1.00mm}
\pgfcircle[fill]{\pgfxy(20.00,20.00)}{1.00mm}
\pgfcircle[stroke]{\pgfxy(20.00,20.00)}{1.00mm}
\pgfmoveto{\pgfxy(40.00,50.00)}\pgflineto{\pgfxy(40.00,40.00)}\pgfstroke
\pgfmoveto{\pgfxy(40.00,50.00)}\pgflineto{\pgfxy(50.00,40.00)}\pgfstroke
\pgfmoveto{\pgfxy(50.00,40.00)}\pgflineto{\pgfxy(60.00,30.00)}\pgfstroke
\pgfmoveto{\pgfxy(60.00,30.00)}\pgflineto{\pgfxy(70.00,20.00)}\pgfstroke
\pgfmoveto{\pgfxy(60.00,30.00)}\pgflineto{\pgfxy(50.00,20.00)}\pgfstroke
\pgfmoveto{\pgfxy(50.00,40.00)}\pgflineto{\pgfxy(40.00,30.00)}\pgfstroke
\pgfmoveto{\pgfxy(40.00,50.00)}\pgflineto{\pgfxy(30.00,40.00)}\pgfstroke
\pgfmoveto{\pgfxy(30.00,40.00)}\pgflineto{\pgfxy(30.00,30.00)}\pgfstroke
\pgfmoveto{\pgfxy(30.00,30.00)}\pgflineto{\pgfxy(40.00,20.00)}\pgfstroke
\pgfmoveto{\pgfxy(30.00,30.00)}\pgflineto{\pgfxy(20.00,20.00)}\pgfstroke
\pgfputat{\pgfxy(32.00,39.00)}{\pgfbox[bottom,left]{\fontsize{11.38}{13.66}\selectfont 7}}
\pgfputat{\pgfxy(22.00,19.00)}{\pgfbox[bottom,left]{\fontsize{11.38}{13.66}\selectfont 8}}
\pgfputat{\pgfxy(52.00,19.00)}{\pgfbox[bottom,left]{\fontsize{11.38}{13.66}\selectfont 0}}
\pgfputat{\pgfxy(62.00,29.00)}{\pgfbox[bottom,left]{\fontsize{11.38}{13.66}\selectfont 10}}
\pgfputat{\pgfxy(42.00,19.00)}{\pgfbox[bottom,left]{\fontsize{11.38}{13.66}\selectfont 4}}
\pgfputat{\pgfxy(72.00,19.00)}{\pgfbox[bottom,left]{\fontsize{11.38}{13.66}\selectfont 5}}
\pgfputat{\pgfxy(52.00,39.00)}{\pgfbox[bottom,left]{\fontsize{11.38}{13.66}\selectfont 6}}
\pgfputat{\pgfxy(42.00,29.00)}{\pgfbox[bottom,left]{\fontsize{11.38}{13.66}\selectfont 2}}
\pgfputat{\pgfxy(43.00,49.00)}{\pgfbox[bottom,left]{\fontsize{11.38}{13.66}\selectfont 9}}
\pgfputat{\pgfxy(32.00,29.00)}{\pgfbox[bottom,left]{\fontsize{11.38}{13.66}\selectfont 1}}
\pgfputat{\pgfxy(42.00,39.00)}{\pgfbox[bottom,left]{\fontsize{11.38}{13.66}\selectfont 3}}
\pgfputat{\pgfxy(20.00,49.00)}{\pgfbox[bottom,left]{\fontsize{11.38}{13.66}\selectfont $T$}}
\pgfcircle[fill]{\pgfxy(110.00,50.00)}{1.00mm}
\pgfcircle[stroke]{\pgfxy(110.00,50.00)}{1.00mm}
\pgfcircle[fill]{\pgfxy(100.00,40.00)}{1.00mm}
\pgfcircle[stroke]{\pgfxy(100.00,40.00)}{1.00mm}
\pgfcircle[fill]{\pgfxy(110.00,40.00)}{1.00mm}
\pgfcircle[stroke]{\pgfxy(110.00,40.00)}{1.00mm}
\pgfcircle[fill]{\pgfxy(120.00,40.00)}{1.00mm}
\pgfcircle[stroke]{\pgfxy(120.00,40.00)}{1.00mm}
\pgfcircle[fill]{\pgfxy(110.00,30.00)}{1.00mm}
\pgfcircle[stroke]{\pgfxy(110.00,30.00)}{1.00mm}
\pgfcircle[fill]{\pgfxy(100.00,30.00)}{1.00mm}
\pgfcircle[stroke]{\pgfxy(100.00,30.00)}{1.00mm}
\pgfmoveto{\pgfxy(110.00,50.00)}\pgflineto{\pgfxy(110.00,40.00)}\pgfstroke
\pgfmoveto{\pgfxy(110.00,50.00)}\pgflineto{\pgfxy(120.00,40.00)}\pgfstroke
\pgfmoveto{\pgfxy(120.00,40.00)}\pgflineto{\pgfxy(110.00,30.00)}\pgfstroke
\pgfmoveto{\pgfxy(110.00,50.00)}\pgflineto{\pgfxy(100.00,40.00)}\pgfstroke
\pgfmoveto{\pgfxy(100.00,40.00)}\pgflineto{\pgfxy(100.00,30.00)}\pgfstroke
\pgfputat{\pgfxy(102.00,39.00)}{\pgfbox[bottom,left]{\fontsize{11.38}{13.66}\selectfont 7}}
\pgfputat{\pgfxy(122.00,39.00)}{\pgfbox[bottom,left]{\fontsize{11.38}{13.66}\selectfont 6}}
\pgfputat{\pgfxy(112.00,29.00)}{\pgfbox[bottom,left]{\fontsize{11.38}{13.66}\selectfont 2}}
\pgfputat{\pgfxy(113.00,49.00)}{\pgfbox[bottom,left]{\fontsize{11.38}{13.66}\selectfont 9}}
\pgfputat{\pgfxy(102.00,29.00)}{\pgfbox[bottom,left]{\fontsize{11.38}{13.66}\selectfont 1}}
\pgfputat{\pgfxy(112.00,39.00)}{\pgfbox[bottom,left]{\fontsize{11.38}{13.66}\selectfont 3}}
\pgfputat{\pgfxy(90.00,49.00)}{\pgfbox[bottom,left]{\fontsize{11.38}{13.66}\selectfont $\MD(T)$}}
\pgfsetlinewidth{1.20mm}\pgfmoveto{\pgfxy(75.00,35.00)}\pgflineto{\pgfxy(85.00,35.00)}\pgfstroke
\pgfmoveto{\pgfxy(85.00,35.00)}\pgflineto{\pgfxy(82.20,35.70)}\pgflineto{\pgfxy(82.20,34.30)}\pgflineto{\pgfxy(85.00,35.00)}\pgfclosepath\pgffill
\pgfmoveto{\pgfxy(85.00,35.00)}\pgflineto{\pgfxy(82.20,35.70)}\pgflineto{\pgfxy(82.20,34.30)}\pgflineto{\pgfxy(85.00,35.00)}\pgfclosepath\pgfstroke
\end{pgfpicture}%
\else
  \setlength{\unitlength}{1bp}%
  \begin{picture}(308.98, 107.72)(0,0)
  \put(0,0){\includegraphics{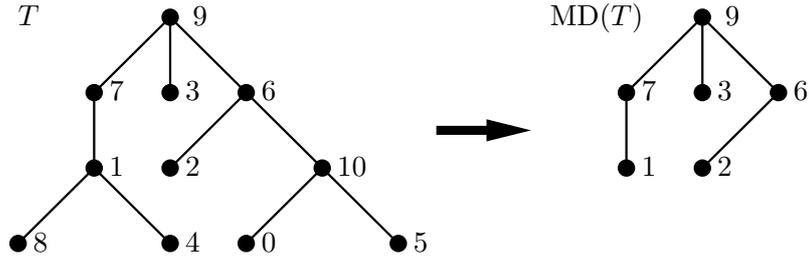}}
  \put(42.52,64.81){\fontsize{11.38}{13.66}\selectfont 7}
  \put(14.17,8.12){\fontsize{11.38}{13.66}\selectfont 8}
  \put(99.21,8.12){\fontsize{11.38}{13.66}\selectfont 0}
  \put(127.56,36.46){\fontsize{11.38}{13.66}\selectfont 10}
  \put(70.87,8.12){\fontsize{11.38}{13.66}\selectfont 4}
  \put(155.91,8.12){\fontsize{11.38}{13.66}\selectfont 5}
  \put(99.21,64.81){\fontsize{11.38}{13.66}\selectfont 6}
  \put(70.87,36.46){\fontsize{11.38}{13.66}\selectfont 2}
  \put(73.70,93.16){\fontsize{11.38}{13.66}\selectfont 9}
  \put(42.52,36.46){\fontsize{11.38}{13.66}\selectfont 1}
  \put(70.87,64.81){\fontsize{11.38}{13.66}\selectfont 3}
  \put(8.50,93.16){\fontsize{11.38}{13.66}\selectfont $T$}
  \put(240.94,64.81){\fontsize{11.38}{13.66}\selectfont 7}
  \put(297.64,64.81){\fontsize{11.38}{13.66}\selectfont 6}
  \put(269.29,36.46){\fontsize{11.38}{13.66}\selectfont 2}
  \put(272.13,93.16){\fontsize{11.38}{13.66}\selectfont 9}
  \put(240.94,36.46){\fontsize{11.38}{13.66}\selectfont 1}
  \put(269.29,64.81){\fontsize{11.38}{13.66}\selectfont 3}
  \put(206.93,93.16){\fontsize{11.38}{13.66}\selectfont $\MD(T)$}
  \end{picture}%
\fi
\caption{The maximal decreasing subtree of the ordered labeled tree $T$}
\label{fig:tree}
\end{figure}

In this paper we present a formula for $|\Ord_{n,k}|$, which makes a refined enumeration of $\Ord_n$, or a generalization of equation~\eqref{eq:cat}. Note that a similar refinement for the rooted (unordered) labeled trees was done before (see \cite{SS12}), but the ordered case is more complicated and has quite different features.

\section{Main results}

From now on we will consider labeled trees only. So we will omit the word ``labeled".
Recall that $\Ord_{n,k}$ is the set of ordered trees on $[0,n]$ with its maximal decreasing ordered subtree having $k$ edges. Let $\Z_{n,k}$ be the set of ordered trees on $[0,n]$ attached additional $(n-k)$ increasing leaves to decreasing tree with $k$ edges. Note that the set $\Z_{n,k}$ first appeared in the Ph.D.~Thesis~\cite[p.\,46]{Dra08} of Drake. Let $\F_{n,k}$ be the set of \seo{(non-ordered)} forests on $[n]:=\{1,2,\ldots,n\}$ consisting of $k$ ordered trees, where the $k$ roots are not ordered. In Figure~\ref{fig:forest}, the first two forests are the same, but the third one is a different forest in $\F_{4,2}$.

\begin{figure}[h]
\centering
\ifpdf
\begin{pgfpicture}{-3.00mm}{-3.86mm}{106.00mm}{14.14mm}
\pgfsetxvec{\pgfpoint{1.00mm}{0mm}}
\pgfsetyvec{\pgfpoint{0mm}{1.00mm}}
\color[rgb]{0,0,0}\pgfsetlinewidth{0.30mm}\pgfsetdash{}{0mm}
\pgfcircle[fill]{\pgfxy(0.00,0.00)}{1.00mm}
\pgfcircle[stroke]{\pgfxy(0.00,0.00)}{1.00mm}
\pgfcircle[fill]{\pgfxy(10.00,0.00)}{1.00mm}
\pgfcircle[stroke]{\pgfxy(10.00,0.00)}{1.00mm}
\pgfcircle[fill]{\pgfxy(5.00,10.00)}{1.00mm}
\pgfcircle[stroke]{\pgfxy(5.00,10.00)}{1.00mm}
\pgfcircle[fill]{\pgfxy(20.00,10.00)}{1.00mm}
\pgfcircle[stroke]{\pgfxy(20.00,10.00)}{1.00mm}
\pgfmoveto{\pgfxy(5.00,10.00)}\pgflineto{\pgfxy(0.00,0.00)}\pgfstroke
\pgfmoveto{\pgfxy(5.00,10.00)}\pgflineto{\pgfxy(10.00,0.00)}\pgfstroke
\pgfcircle[stroke]{\pgfxy(5.00,10.00)}{2.00mm}
\pgfcircle[stroke]{\pgfxy(20.00,10.00)}{2.00mm}
\pgfcircle[fill]{\pgfxy(50.00,0.00)}{1.00mm}
\pgfcircle[stroke]{\pgfxy(50.00,0.00)}{1.00mm}
\pgfcircle[fill]{\pgfxy(60.00,0.00)}{1.00mm}
\pgfcircle[stroke]{\pgfxy(60.00,0.00)}{1.00mm}
\pgfcircle[fill]{\pgfxy(55.00,10.00)}{1.00mm}
\pgfcircle[stroke]{\pgfxy(55.00,10.00)}{1.00mm}
\pgfmoveto{\pgfxy(55.00,10.00)}\pgflineto{\pgfxy(50.00,0.00)}\pgfstroke
\pgfmoveto{\pgfxy(55.00,10.00)}\pgflineto{\pgfxy(60.00,0.00)}\pgfstroke
\pgfcircle[stroke]{\pgfxy(55.00,10.00)}{2.00mm}
\pgfcircle[fill]{\pgfxy(40.00,10.00)}{1.00mm}
\pgfcircle[stroke]{\pgfxy(40.00,10.00)}{1.00mm}
\pgfcircle[stroke]{\pgfxy(40.00,10.00)}{2.00mm}
\pgfcircle[fill]{\pgfxy(90.00,0.00)}{1.00mm}
\pgfcircle[stroke]{\pgfxy(90.00,0.00)}{1.00mm}
\pgfcircle[fill]{\pgfxy(100.00,0.00)}{1.00mm}
\pgfcircle[stroke]{\pgfxy(100.00,0.00)}{1.00mm}
\pgfcircle[fill]{\pgfxy(95.00,10.00)}{1.00mm}
\pgfcircle[stroke]{\pgfxy(95.00,10.00)}{1.00mm}
\pgfmoveto{\pgfxy(95.00,10.00)}\pgflineto{\pgfxy(90.00,0.00)}\pgfstroke
\pgfmoveto{\pgfxy(95.00,10.00)}\pgflineto{\pgfxy(100.00,0.00)}\pgfstroke
\pgfcircle[stroke]{\pgfxy(95.00,10.00)}{2.00mm}
\pgfcircle[fill]{\pgfxy(80.00,10.00)}{1.00mm}
\pgfcircle[stroke]{\pgfxy(80.00,10.00)}{1.00mm}
\pgfcircle[stroke]{\pgfxy(80.00,10.00)}{2.00mm}
\pgfputat{\pgfxy(13.00,-1.00)}{\pgfbox[bottom,left]{\fontsize{11.38}{13.66}\selectfont \makebox[0pt]{$1$}}}
\pgfputat{\pgfxy(63.00,-1.00)}{\pgfbox[bottom,left]{\fontsize{11.38}{13.66}\selectfont \makebox[0pt]{$1$}}}
\pgfputat{\pgfxy(87.00,-1.00)}{\pgfbox[bottom,left]{\fontsize{11.38}{13.66}\selectfont \makebox[0pt]{$1$}}}
\pgfputat{\pgfxy(24.00,9.00)}{\pgfbox[bottom,left]{\fontsize{11.38}{13.66}\selectfont \makebox[0pt]{$2$}}}
\pgfputat{\pgfxy(44.00,9.00)}{\pgfbox[bottom,left]{\fontsize{11.38}{13.66}\selectfont \makebox[0pt]{$2$}}}
\pgfputat{\pgfxy(84.00,9.00)}{\pgfbox[bottom,left]{\fontsize{11.38}{13.66}\selectfont \makebox[0pt]{$2$}}}
\pgfputat{\pgfxy(9.00,9.00)}{\pgfbox[bottom,left]{\fontsize{11.38}{13.66}\selectfont \makebox[0pt]{$3$}}}
\pgfputat{\pgfxy(59.00,9.00)}{\pgfbox[bottom,left]{\fontsize{11.38}{13.66}\selectfont \makebox[0pt]{$3$}}}
\pgfputat{\pgfxy(99.00,9.00)}{\pgfbox[bottom,left]{\fontsize{11.38}{13.66}\selectfont \makebox[0pt]{$3$}}}
\pgfputat{\pgfxy(3.00,-1.00)}{\pgfbox[bottom,left]{\fontsize{11.38}{13.66}\selectfont \makebox[0pt]{$4$}}}
\pgfputat{\pgfxy(53.00,-1.00)}{\pgfbox[bottom,left]{\fontsize{11.38}{13.66}\selectfont \makebox[0pt]{$4$}}}
\pgfputat{\pgfxy(103.00,-1.00)}{\pgfbox[bottom,left]{\fontsize{11.38}{13.66}\selectfont \makebox[0pt]{$4$}}}
\pgfputat{\pgfxy(30.00,4.00)}{\pgfbox[bottom,left]{\fontsize{11.38}{13.66}\selectfont \makebox[0pt]{$=$}}}
\pgfputat{\pgfxy(70.00,4.00)}{\pgfbox[bottom,left]{\fontsize{11.38}{13.66}\selectfont \makebox[0pt]{$\neq$}}}
\end{pgfpicture}%
\else
  \setlength{\unitlength}{1bp}%
  \begin{picture}(308.98, 51.02)(0,0)
  \put(0,0){\includegraphics{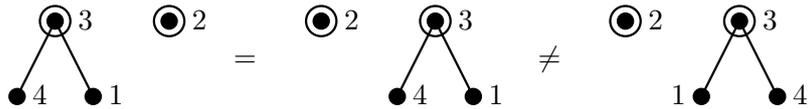}}
  \put(45.35,8.12){\fontsize{11.38}{13.66}\selectfont \makebox[0pt]{$1$}}
  \put(187.09,8.12){\fontsize{11.38}{13.66}\selectfont \makebox[0pt]{$1$}}
  \put(255.12,8.12){\fontsize{11.38}{13.66}\selectfont \makebox[0pt]{$1$}}
  \put(76.54,36.46){\fontsize{11.38}{13.66}\selectfont \makebox[0pt]{$2$}}
  \put(133.23,36.46){\fontsize{11.38}{13.66}\selectfont \makebox[0pt]{$2$}}
  \put(246.61,36.46){\fontsize{11.38}{13.66}\selectfont \makebox[0pt]{$2$}}
  \put(34.02,36.46){\fontsize{11.38}{13.66}\selectfont \makebox[0pt]{$3$}}
  \put(175.75,36.46){\fontsize{11.38}{13.66}\selectfont \makebox[0pt]{$3$}}
  \put(289.13,36.46){\fontsize{11.38}{13.66}\selectfont \makebox[0pt]{$3$}}
  \put(17.01,8.12){\fontsize{11.38}{13.66}\selectfont \makebox[0pt]{$4$}}
  \put(158.74,8.12){\fontsize{11.38}{13.66}\selectfont \makebox[0pt]{$4$}}
  \put(300.47,8.12){\fontsize{11.38}{13.66}\selectfont \makebox[0pt]{$4$}}
  \put(93.54,22.29){\fontsize{11.38}{13.66}\selectfont \makebox[0pt]{$=$}}
  \put(206.93,22.29){\fontsize{11.38}{13.66}\selectfont \makebox[0pt]{$\neq$}}
  \end{picture}%
\fi
\caption{Forests in $\F_{4,2}$}
\label{fig:forest}
\end{figure}

Define the numbers
\begin{align*}
o(n,k) &= |\Ord_{n,k}|,\\
z(n,k) &= |\Z_{n,k}|,\\
f(n,k) &= |\F_{n,k}|.
\end{align*}
We will show that an ordered tree can be ``decomposed" into an ordered tree in $\cup_{n,k}\Z_{n,k}$ and a forest in $\cup_{n,k}\F_{n,k}$. Thus it is crucial to count the numbers $z(n,k)$ and $f(n,k)$.

\begin{lem} \label{lem:znk}
The numbers $z(n,k)$ satisfy the recursion:
\begin{align}\label{eq:recurZ}
z(n,k) = n \cdot z(n-1,k) + (n+k-1)\cdot z(n-1,k-1)
\quad \text{for $1\le k <n$}
\end{align}
with the following boundary conditions:
\begin{align}
z(n,n) &= (2n-1)!! \quad \text{for $n \ge 0$} \label{eq:Z3}\\
z(n,k) &=0 \quad \text{for $n<k$ or $k<0$} \label{eq:Z2},
\end{align}
where $(2n-1)!!$ is defined by $(2n-1)=(2n-1)(2n-3)\cdots 3\cdot 1$.
\end{lem}

\begin{proof}
Consider a tree $Z$ in $\Z_{n,k}$. The tree $Z$ with $n+1$ vertices consists of its maximal decreasing tree with $k+1$ vertices and the number of increasing leaves is $n-k$.
Note that the vertex $0$ is always contained in $\MD(Z)$.

If the vertex $0$ is a leaf of $Z$, consider the tree $Z'$ by deleting the leaf $0$ from $Z$.
The number of vertices in $Z'$ and $\MD(Z')$ are $n$ and $k$, respectively.
So the number of possible trees $Z'$ is $z(n-1,k-1)$.
Since we cannot attach the vertex $0$ to $n-k$ increasing leaves in recovering $Z$, there are $(2n-1)-(n-k)$ ways of recovering $Z$.
Thus the number of $Z$ with the leaf $0$ is
$$(n+k-1) \cdot z(n-1,k-1).$$

If the vertex $0$ is not a leaf of $Z$, then the vertex $0$ has at least one increasing leaf.
Let the vertex $\ell$ be the leftmost leaf of the vertex $0$ and consider the tree $Z''$ obtained by deleting the leaf $\ell$ from $Z$.
The number of vertices in $Z''$ and $\MD(Z'')$ are $n$ and $k+1$, respectively.
So the number of possible trees $Z''$ is $z(n-1,k)$.
To recover $Z$ is to relabel $Z''$ with $[0,n] \setminus \set {\ell}$ and to attach the vertex $\ell$ to the vertex $0$.
Since the number $\ell$ may be the number from $1$ to $n$,
the number of $Z$ without the leaf $0$ is
$$n \cdot z(n-1,k),$$
which completes the proof of recursion \eqref{eq:recurZ}.

Since $\Z(n,n)$ is the set of decreasing ordered trees on $[0,n]$, the equation \eqref{eq:Z3} holds \cite{Kl97} with the convention $(-1)!! =1$.
For $n<k$ or $k<0$, $\Z_{n,k}$ should be empty, so the equation \eqref{eq:Z2} also holds.
\end{proof}

\begin{lem} \label{lem:fnk}
For $0\le k\le n$, we have
\begin{align}
f(n,k) = {n \choose k} \,k\, (n+1)(n+2) \cdots (2n-k-1) \label{eq:num_f}
\end{align}
with $f(0,0) = 1.$
\end{lem}

\begin{proof}
Consider a forest $F$ in $\F_{n,k}$.
The forest $F$ consists of (non-ordered) $k$ ordered trees $O_1,\ldots,O_k$ with roots $r_1, r_2, \ldots, r_k$, where $r_1 < r_2 < \cdots < r_k$.
The number of ways for choosing roots $r_1, r_2, \cdots, r_k$ from $[n]$ is equal to $n \choose k$.
From the \emph{reverse Pr\"ufer algorithm (RP Algorithm)} in \cite{SS07}, the number of ways for adding $n-k$ vertices successively to $k$ roots $r_1, r_2, \cdots, r_k$ is equal to
$$
k (n+1) (n+2) \cdots (2n-k-1)
$$
for $ 0<  k < n$, thus the equation \eqref{eq:num_f} holds.
By definition, $\F(0,0)$ is the set of the empty forest. So $f(0,0) =1$.
\end{proof}

Since the number $z(n,k)$ is determined by the recurrence relation \eqref{eq:recurZ} in Lemma~\ref{lem:znk}, we can count the number $o(n,k)$ with the
following theorem.
\begin{thm} \label{thm:main}
We have
\begin{align}
o(n,k) = \sum_{k \le m \le n} \binom{n+1}{m+1} \, z(m,k) \,\frac{m-k}{n-k} (n-k)^{(n-m)} \quad \text{for\quad$0 \le k<n$,} \label{eq:num_o}
\end{align}
and
$o(n,n) = (2n-1)!!$,
where $n^{(k)}$ is a rising factorial.
%where $n^{(k)} := n (n+1) \cdots (n+k-1)$ is a rising factorial.

\end{thm}

\begin{proof}
Given an ordered tree $T$ in $\Ord_{n,k}$, let $Z$ be the subtree of $T$ consisting of $\MD(T)$ and its increasing edges.
If the number of vertices of $Z$ is $m+1$, then $Z$ is a subtree of $T$ with $(m-k)$ increasing leaves.
Also, the induced subgraph $Y$ of $T$ generated by the $(n-k)$ vertices not belonging to $\MD(T)$ is a (non-ordered) forest consisting of $(m-k)$ ordered trees whose roots are only increasing leaves of $Z$.

Now let us count the number of ordered trees $T \in \Ord_{n,k}$ with $\abs{V(Z)}=m+1$ where $V(Z)$ is the set of vertices in $Z$.
First of all, the number of ways for selecting a set $V(Z) \subset [0,n]$  is equal to $n+1 \choose m+1$.
By attaching $(m-k)$ increasing leaves to a decreasing tree with $k$ edges, we can make an ordered trees on $V(Z)$.
There are exactly $z(m,k)$ ways for making such an ordered subtree on $V(Z)$.
By the definition of $\F_{n,k}$ and Lemma~\ref{lem:fnk}, the number of ways for constructing the other parts on $V(T)\setminus V(Z)$ is equal to
$$\left. f(n-k, m-k) \middle/  {n-k \choose m-k} \right. =  \frac{m-k}{n-k} (n-k)^{(n-m)}.$$
Since the range of $m$ is $k\le m \le n$, the equation \eqref{eq:num_o} holds.

Finally, $\Ord(n,n)$ is the set of decreasing ordered trees on $[0,n]$, so
$$o(n,n) = z(n,n)=(2n-1)!!$$
holds for $n \ge 0$.
\end{proof}

\section{Remark}
Due to Theorem~\ref{thm:main}, we can calculate $o(n,k)$ for all $n$, $k$.
%However a problem is that a closed form, a recurrence relation, or a generating function of $o(n,k)$ are not found yet.
However a closed form, a recurrence relation, or a generating function of $o(n,k)$ have not been found yet.
The following might be a direction for solving the problem:

Shor \cite{Sho95} showed that the number $r(n,k)$, which is the number of rooted trees on $[n]$ with $k$ improper edges, satisfies
$$r(n,k) = (n-1)\, r(n-1,k) + (n+k-2)\, r(n-1,k-1),$$
where an edge $(u,v)$ is called \emph{improper} if $u$ is the endpoint closer to root and $u$ has a larger label than some descendant of $v$.
Zeng \cite{CWY11, Zen99} found that the generating function for $\left\{r(n,k)\right\}_{k=0}^n$ is the Ramanujan polynomial $R_n(x)$, which is defined by
$$R_{n+1} (x) = n (1+x) R_n(x) + x^2 R'_n (x); \quad R_1(x) = 1.$$
Drake \cite[p.\,46]{Dra08} observed that $z(n,k) = r(n+1,k)$ for all $k\le n$, by using the generating function method.
Actually, $z(n,k)$ and $r(n+1,k)$ satisfy the same recursion and initial conditions, so we are able to construct a recursive bijection between these two objects. With this point of view, it would be interesting to find a certain set of rooted trees of cardinality $o(n,k)$.

\section*{Acknowledgment}

This research was supported by Basic Science Research Program through the National Research Foundation of Korea (NRF) funded by the Ministry of Education, Science and Technology (2011-0008683, 2012R1A1A1014154).

%%%%%%%%%%%%%%%%%%%%%%%%%%%%%%%%%%%%%%%%%%%%%%%%%%%%%%%%%%%%%%%%%%%%%%%%%%%%%%

%\bibliographystyle{plain}
%\bibliography{refer}

\end{document}